\documentclass{article}
\usepackage{hyperref}
\usepackage{natbib}

\usepackage{graphicx}
\usepackage{graphics}

\usepackage{tikz}
\usepackage{amsmath}
\usepackage{pgfplots}
\usetikzlibrary{patterns}
\usetikzlibrary{decorations.pathreplacing}
\usetikzlibrary{arrows.meta}

\usepackage{amsfonts}
\usepackage{amssymb}
\usepackage{enumerate}
\usepackage{xcolor}
\usepackage[normalem]{ulem}
\usepackage{comment}

\DeclareMathOperator{\argmax}{arg\,max}

\newtheorem{theorem}{Theorem}[section]

\newtheorem{definition}[theorem]{Definition}
\newtheorem{example}[theorem]{Example}

\newtheorem{lemma}[theorem]{Lemma}
\newtheorem{assumption}[theorem]{Assumption}

\newtheorem{remark}[theorem]{Remark}

\newenvironment{proof}[1][Proof]{\textbf{#1.} }{\ \rule{0.5em}{0.5em}}

\newcommand{\E}{{\rm \bf E}}

\newcommand{\prob}{{\rm \bf P}}

\newcommand{\supp}{{\rm supp}}
\newcommand{\dN}{{\bf N}}

\newcommand{\dR}{{\bf R}}

\newcommand{\calE}{{\cal E}}

\newcommand{\calH}{{\cal H}}

\newcommand{\ep}{\varepsilon}

\newcounter{figurecounter}
\newcounter{figure:table}
\newcounter{figure:table1}
\newcounter{figure:game}
\newcommand{\black}{\color{black}}

\usepackage[textwidth=30mm]{todonotes}

\title{Undiscounted Equilibrium in Positive Recursive Absorbing Games with Non-Rectangular Absorption Structure%
\thanks{
The authors thank J\'anos Flesch for commenting on an earlier version of the paper.
Solan acknowledges the support of the Israel Science Foundation grant \#211/22.}}

\author{Eilon Solan\thanks{School of Mathematical Sciences, Tel Aviv University, Tel
Aviv 69978, Israel. E-mail: \textsf{eilons@tauex.tau.ac.il}.},
and Nicolas Vieille\thanks{Department of Economics and Decision Sciences, HEC Paris, 1, rue de
la Lib\'{e}ration, 78 351 Jouy-en-Josas, France. E-mail: \textsf{vieille@hec.fr}.}}

\date{\today}

\begin{document}

\maketitle

\begin{abstract}
An \emph{absorbing game} is a stochastic game with a single nonabsorbing state.
Such a game is called \emph{recursive} if all players receive a payoff of $0$
in the nonabsorbing state, and \emph{positive} if all payoffs in absorbing states are positive.
An action profile is \emph{nonabsorbing} if, when it is played, the game remains in the nonabsorbing state with probability 1.
The set of nonabsorbing action profiles can be partitioned into the connected components of an undirected graph, whose vertices are these profiles, with two vertices joined by an edge whenever the corresponding profiles differ in the action of a single player.
A connected component is said to be \emph{rectangular} if it is the Cartesian product of subsets of the players’ action sets.

We prove that every positive recursive absorbing game whose nonabsorbing components are all non-rectangular admits an undiscounted equilibrium payoff.
\end{abstract}



\section{Introduction}

One of the main open problems in game theory to date is whether every multiplayer stochastic game admits an undiscounted equilibrium payoff.
This open problem was answered affirmatively for two-player zero-sum games (\cite{mertens1981stochastic}),
for two-player non-zero-sum games (\cite{vieille2000two, vieille2000recursive}),
and for various classes of stochastic games with more than two players (see, e.g., \cite{Solan1999ThreePlayerAbsorbing, SolanVieille2001QuittingGames}, \cite{simon2007structure, simon2012topological}, \cite{flesch2007stochastic, flesch2008product, flesch2009stochastic}, and \cite{SolanSolanSolan2020JointlyControlled}.

Absorbing games are stochastic games with a single nonabsorbing state.
This class of games was introduced by \cite{kohlberg1974repeated},
who proved the existence of the limit value in two-player zero-sum absorbing games,
and has since been instrumental in the study of stochastic games.
Indeed, the result of \cite{kohlberg1974repeated} was extended to all stochastic games by \cite{mertens1981stochastic}.
Similarly, 
the existence of an undiscounted equilibrium payoff in two-player non-zero-sum absorbing games,
proved by \cite{vrieze1989equilibria},
was extended to all two-player stochastic games by \cite{vieille2000two, vieille2000recursive}.
And the characterization of the value in zero-sum absorbing games (\cite{laraki2010explicit}) was extended to all zero-sum stochastic games by \cite{attia2019formula}.

Some results that have been proven for absorbing games have so far resisted extension to stochastic games, e.g.,
the existence of an undiscounted equilibrium in three-player absorbing games (\cite{Solan1999ThreePlayerAbsorbing}) and in absorbing team games (\cite{Solan2000AbsorbingTeamGames}),
the existence of an undiscounted normal-form correlated equilibrium in multiplayer absorbing games (\cite{SolanVohra2002CorrelatedAbsorbingGames}),
and weak approachability in absorbing games (\cite{ragel2024weak}).
Some subclasses of absorbing games, such as quitting games (\cite{SolanVieille2001QuittingGames}) and positive recursive absorbing games (\cite{SolanSolan2021SunspotEquilibrium}) were also studied and provided tools and useful insights for the study of stochastic games.

\paragraph{Main result and main idea behind the proof.}

In this paper, we focus on multiplayer absorbing games for which the stage payoff of all players in the nonabsorbing state is $0$,
and all absorbing payoffs are positive. We prove that an undiscounted equilibrium payoff exists in 
these games,
under an additional restriction on the transition structure.

Our proof technique is an extension of the approach of \cite{SolanVieille2001QuittingGames}.
Assuming the absorbing payoffs are bounded from above by $1$,
and denoting by $I$ the set of players,
we define a function $f$ from a subset of $[0,1]^I$ to itself, with the following property. Given $w$ and $\ep>0$, there is an integer $T(w)$ and a strategy profile  $\sigma(w)$ such that (i) $\sigma(w)$ is an $\ep$-equilibrium in the $T(w)$-stage version of the game that ends with a terminal payoff of $w$ in case the game did not absorb before, (ii) the probability given $\sigma(w)$ that the game absorbs is positive and (iii) the payoff induced by $\sigma(w)$ is $\ep$-close to $f(w)$.

We use a finite orbit of $f$ to construct a sequence $(w^{(k)})_{k=0}^{K_0}$ such that $w^{(k)} = f(w^{(k+1)})$, with $w^{(K_0)}= (1,\ldots, 1)$.
Finally, we define a strategy profile $\sigma^*$
by letting the players follow $\sigma(w^{(0)})$ for $T(w^{(0)})$ stages, then $\sigma(w^{(1)})$ for $N(w^{(1)})$ stages, and so on. 

\paragraph{Significance of the paper.}

Beyond advancing the long-standing open problem of whether every multiplayer stochastic game admits an undiscounted equilibrium payoff, the paper makes several conceptual contributions.

First, we extend the dynamical-system approach to stochastic games, initiated by \cite{SolanVieille2001QuittingGames} for a restricted class of quitting games, to a broad class of absorbing games.


Second, together with \cite{SolanSolanSolan2020JointlyControlled}, our result delineates where the main difficulty in establishing equilibrium in recursive absorbing games lies. The non-rectangularity condition identifies a large, easily checkable class of games for which the equilibrium problem is now resolved.

Third, the ideas developed here can be used to study more general stochastic games and other solution concepts. In a companion paper \cite{solanvieillepublic}, building on and refining the techniques introduced here, we show that in the presence of a public correlation device, 
every positive recursive stochastic game (not necessarily absorbing) admits an undiscounted equilibrium payoff. Thus, the present paper provides the core building block for a more general theory.

Fourth, our arguments can be used to substantially simplify existing proofs in the literature, such as the proof of the main result of \cite{SolanVieille2001QuittingGames}, thereby clarifying the underlying structure of those results.
\black






\paragraph{Structure of the paper.}

The model and the main result are presented in Section~\ref{sec:model}.
Section~\ref{section:basic} presents some new and old concepts that will be used in the proof.
The proof appears in Section~\ref{section:proof},
and open problems are discussed in Section~\ref{section:discussion}.

\section{Model and Main Result}
\label{sec:model}

\subsection{The Model}

\begin{definition}[Recursive absorbing game]
A \emph{recursive absorbing game} is a tuple 
$\Gamma = (I, (A_i)_{i \in I}, (r_i)_{i \in I}, p)$,
where
\begin{itemize}
\item 
$I$ is a finite set of players.
\item 
$A_i$ is a finite set of actions, for each player $i \in I$.
Denote the set of action profiles by $A := \prod_{i \in I} A_i$.
\item 
$r_i : A \to \dR$ is player~$i$'s absorbing payoff function,
for each $i \in I$.
\item 
$p : A \to [0,1]$ is the absorption probability function.
\end{itemize}
\end{definition}

The game is played as follows.
At every stage $n \in \dN$, 
each player $i$ selects an action $a_i^n \in A_i$ if the game was not absorbed before.
With probability $p(a^n)$, where $a^n := (a^n_i)_{i \in I}$,
the game is absorbed, and each player $i \in I$ receives the terminal payoff $r_i(a^n)$.
With probability $1-p(a^n)$ the game is not absorbed.

\begin{definition}[Positive recursive absorbing game]
A recursive absorbing game $\Gamma = (I, (A_i)_{i \in I}, (r_i)_{i \in I}, p)$ is \emph{positive}
if $r_i(a) > 0$, for every $i \in I$ and every $a \in A$.
\end{definition}

For convenience, 
when the game is positive and recursive, we will assume that the payoff function $r$ is bounded between $0$ and $1$.

We denote by $\theta \in \dN\cup \{+\infty\}$ the stage at which absorption occurs. The game terminates in stage $\theta$, yet it will be convenient to assume that players choose actions in all stages $n\in \dN$, even after $\theta$.

The set of \emph{histories}
is the set $H := \bigcup_{n=0}^\infty A^n$.
A history $(a^1,a^2,\dots,a^n)$ is interpreted as the sequence of actions in the first $n$ stages.
A (behavior) \emph{strategy} of player~$i$ is a function 
$\sigma_i : H \to \Delta(A_i)$.
The set of all strategies of player~$i$ is denoted $\Sigma_i$,
and the set of all strategy profiles of all players except $i$ is denoted $\Sigma_{-i} = \prod_{j \neq i} \Sigma_j$.
A \emph{play} is an infinite sequence of action profiles $(a^1,a^2,\dots)$.

Every strategy profile $\sigma = (\sigma_i)_{i \in I}$
induces a probability distribution $\prob_\sigma$ on the set of plays, endowed with the cylinder $\sigma$-algebra.
We denote by $\E_\sigma$ the corresponding expectation operator.

The (expected) \emph{undiscounted payoff} under strategy profile $\sigma$ is
\[ \gamma(\sigma) := \E_\sigma[r(a^\theta){1_{\theta< +\infty}}] \in [0,1]^I. \]

Note that for action profiles $a\in A$ with $p(a) = 0$,
the value $r(a)$ does not affect the undiscounted payoff.

For every vector $b = (b_j)_{j \in I}$
we denote $b_{-i} = (b_j)_{j \neq i}$,
and for every subset $J \subseteq I$
we denote $b_J = (b_j)_{j \in J}$ and $b_{-J} = (b_j)_{j \not\in J}$.

\begin{definition}[Undiscounted equilibrium]
Let $\ep > 0$.
The strategy profile $\sigma^*$ is an \emph{undiscounted $\ep$-equilibrium} if for every $i \in I$ and every $\sigma_i \in \Sigma_i$,
\[ \gamma_i(\sigma^*) \geq \gamma_i(\sigma_i,\sigma^*_{-i}) - \ep.\]
A vector $z \in \dR^I$ is an \emph{undiscounted equilibrium payoff} if $z = \lim_{k \to \infty} \gamma(\sigma^k)$,
for some sequence $(\sigma^k)_{k \in \dN}$ such that 
$\sigma^k$ is a $\frac{1}{k}$-equilibrium,
for each $k \in \dN$.
\end{definition}

\subsection{The Absorption Structure}

\begin{definition}[Set of nonabsorbing action profiles $B$]
The set of nonabsorbing action profiles is
\[ B := \{ a \in A \colon p(a) = 0\}. \]
\end{definition}

\begin{definition}[Connected components of nonabsorbing action profiles]
Consider the graph whose set of vertices is the set of nonabsorbing action profiles $B$,
and there is an edge between $a$ and $a'$ if 
these two action profiles differ in the action of at most one player;
that is,
there is $i \in I$ such that $a_{-i} = a'_{-i}$.
Any connected component of this graph is called a \emph{connected component of nonabsorbing action profiles}.
\end{definition}

For simplicity, 
we will shorten the term ``connected component of nonabsorbing action profiles'' to ``connected component''.
Denote by $(B^\ell)_{l=1}^L$ the collection of connected components, so that $B = \bigcup_{\ell =1}^L B^\ell$.

\begin{definition}[rectangular connected component]
The connected component $B^\ell$ is \emph{rectangular}
if it is a product set $B^\ell = \prod_{i \in I} B^\ell_i$.
\end{definition}

\begin{example}\label{example:2}
In Figure~\ref{fig:a-grid} we depict the absorption structure of a three-player absorbing game,
where Player~1 selects a row,
Player~2 selects a column,
Player~3 selects a matrix,
and absorbing entries are marked with an asterisk.
In this game there are three connected components of nonabsorbing action profiles:
\begin{itemize}
\item 
One contains a single nonabsorbing action profile,
and in particular is rectangular:
$\{(a'_1,a'_2,a''_3)\}$.
\item 
One is rectangular and contains six nonabsorbing action profiles:

$\{a'''_1\} \times \{a''_2,a'''_2\} \times \{a_3,a'_3,a''_3\}$.
\item 
One contains six nonabsorbing action profiles and is not rectangular:

$\{(a_1,a_2),(a_1,a'_2),(a'_1,a_2)\} \times \{a_3,a'_3\}$.
\end{itemize}
\end{example}

\begin{figure}[ht]
  \centering
\newcommand{\rectwidth}{0.8cm}  
\newcommand{\rectheight}{0.8cm} 
\begin{tikzpicture}
  \begin{scope}[x=\rectwidth, y=\rectheight]
  \draw[xstep=\rectwidth, ystep=\rectheight] (0,0) grid (4,4);

    \node[left] at (0,3.5)  {$a_1$};
    \node[left] at (0,2.5)  {$a'_1$};
    \node[left] at (0,1.5)  {$a''_1$};
    \node[left] at (0,0.5)  {$a'''_1$};

    \node[above] at (0.5,4) {$a_2$};
    \node[above] at (1.5,4) {$a'_2$};
    \node[above] at (2.5,4) {$a''_2$};
    \node[above] at (3.5,4) {$a'''_2$};

    \node[above] at (2,4.5) {$a_3$};

    \node at (0.5,0.5) {*};
    \node at (1.5,0.5) {*};
    \node at (0.5,1.5) {*};
    \node at (1.5,1.5) {*};
    \node at (2.5,1.5) {*};
    \node at (3.5,1.5) {*};
    \node at (1.5,2.5) {*};
    \node at (2.5,2.5) {*};
    \node at (3.5,2.5) {*};
    \node at (2.5,3.5) {*};
    \node at (3.5,3.5) {*};
  \end{scope}

  \begin{scope}[x=\rectwidth, y=\rectheight, shift={(5,0)}]
  \draw[xstep=\rectwidth, ystep=\rectheight] (0,0) grid (4,4);

    \node[left] at (0,3.5)  {$a_1$};
    \node[left] at (0,2.5)  {$a'_1$};
    \node[left] at (0,1.5)  {$a''_1$};
    \node[left] at (0,0.5)  {$a'''_1$};

    \node[above] at (0.5,4) {$a_2$};
    \node[above] at (1.5,4) {$a'_2$};
    \node[above] at (2.5,4) {$a''_2$};
    \node[above] at (3.5,4) {$a'''_2$};

    \node[above] at (2,4.5) {$a'_3$};

    \node at (0.5,0.5) {*};
    \node at (1.5,0.5) {*};
    \node at (0.5,1.5) {*};
    \node at (1.5,1.5) {*};
    \node at (2.5,1.5) {*};
    \node at (3.5,1.5) {*};
    \node at (1.5,2.5) {*};
    \node at (2.5,2.5) {*};
    \node at (3.5,2.5) {*};
    \node at (2.5,3.5) {*};
    \node at (3.5,3.5) {*};
  \end{scope}

  \begin{scope}[x=\rectwidth, y=\rectheight, shift={(10,0)}]
  \draw[xstep=\rectwidth, ystep=\rectheight] (0,0) grid (4,4);

    \node[left] at (0,3.5)  {$a_1$};
    \node[left] at (0,2.5)  {$a'_1$};
    \node[left] at (0,1.5)  {$a''_1$};
    \node[left] at (0,0.5)  {$a'''_1$};

    \node[above] at (0.5,4) {$a_2$};
    \node[above] at (1.5,4) {$a'_2$};
    \node[above] at (2.5,4) {$a''_2$};
    \node[above] at (3.5,4) {$a'''_2$};

    \node[above] at (2,4.5) {$a''_3$};

    \node at (0.5,0.5) {*};
    \node at (1.5,0.5) {*};
    \node at (0.5,1.5) {*};
    \node at (1.5,1.5) {*};
    \node at (2.5,1.5) {*};
    \node at (3.5,1.5) {*};
    \node at (0.5,2.5) {*};
    \node at (2.5,2.5) {*};
    \node at (3.5,2.5) {*};
    \node at (0.5,3.5) {*};
    \node at (1.5,3.5) {*};
    \node at (2.5,3.5) {*};
    \node at (3.5,3.5) {*};
  \end{scope}
\end{tikzpicture}
  \caption{The absorption structure of the three-player recursive absorbing game in Example~\ref{example:2}.}
  \label{fig:a-grid}
\end{figure}
\setcounter{figurecounter}{2}

\subsection{Main Result}

It follows from \cite{SolanVieille2001QuittingGames}
that all recursive absorbing games
that satisfy the following restrictive conditions admit an undiscounted equilibrium payoff:
\begin{itemize}
\item Exactly one action profile $\widehat a$ is nonabsorbing.
\item For every player~$i \in I$ and every action profile $b_{-i}\in A_{-i}$, one has 
$r_i(a_i,\widehat a_{-i}) \geq r_i(a_i,b_{-i})$, for each $a_i$ in the set
\[\widehat A_i := \argmax_{b_i\in A_i\colon p(b_i,\widehat{a}_{-i})>0} r_i(b_i,\widehat a_{-i}) \] 
of best absorbing actions against $\widehat a$.
\end{itemize}

\cite{SolanSolanSolan2020JointlyControlled}
proved that all positive recursive absorbing games
that satisfy the following conditions 
admit an undiscounted equilibrium payoff:
\begin{itemize}
\item 
There is exactly one connected component which is rectangular, 
so that the set of nonabsorbing entries is a Cartesian product $\prod_{i \in I} B_i$.
Moreover,
$|B_i| \geq 2$ for at least two players $i \in I$.
\item 
Each player has exactly one action that is not part of any nonabsorbing entry:
$|A_i\setminus B_i|= 1$, for every $i \in I$.
\item 
The probability of absorption in all entries is either $0$ or $1$,
that is, $p(a) \in \{0,1\}$ for every $a \in A.$
\end{itemize}

Our main result shows that an undiscounted equilibrium payoff exists when \emph{all} connected components are not rectangular.

\begin{theorem}[Existence of undiscounted equilibrium]
\label{theorem:uniform}
Every positive recursive absorbing game that has no rectangular connected component admits an undiscounted equilibrium payoff.
\end{theorem}

\begin{remark}[Uniform equilibrium]
Theorem~\ref{theorem:uniform} also applies to the
stronger concept of \emph{uniform equilibrium payoff}, see 
\cite{mertens2015repeated}, Section VII.4.
A vector $z \in \dR^I$ is a uniform equilibrium payoff if for every $\ep > 0$ there is a strategy profile $\sigma$ which is
 an $\ep$-equilibrium in all games with finite horizon,
 and induces a payoff vector within $\ep$ of $z$ in all such games, provided the horizon is sufficiently long.

In positive recursive absorbing games, the expected average payoff $\gamma_n(\sigma)$ over the first $n$ stages is nondecreasing in $n$, and converges to the undiscounted payoff $\gamma(\sigma)$. This implies
that the two concepts of equilibrium payoffs coincide for such games.

Thus, Theorem~\ref{theorem:uniform} implies that every positive recursive absorbing game that has no rectangular connected component admits a uniform equilibrium payoff as well.
\end{remark}

\section{Basic Concepts}
\label{section:basic}


\subsection{Mixed-Action Profiles and Connected Components}

For every mixed action $x_i \in  \Xi_i := \Delta(A_i)$,
denote
\[ \supp(x_i) := \{ a_i \in A_i \colon x_i(a_i) > 0\}. \]
Set $\Xi := \prod_{j \in I} \Xi_j$ 
and 
$\Xi_{-i} := \prod_{j \neq i} \Xi_j$. 
For every mixed action profile $x \in \Xi$ set $\supp(x) := \prod_{i \in I} \supp(x_i)$.

The set of nonabsorbing mixed action profiles whose support is contained in the connected component $B^\ell$ is
\[ X^\ell := \left\{ x \in \Xi \colon \supp(x) \subseteq B^\ell \right\}, \ \ \ \forall \ell \in [L] := \{1,2,\ldots,L\}, \]
and the set of all nonabsorbing mixed action profiles is
\[ X := \bigcup_{l=1}^L X^\ell. \]
We note that each $X^\ell$ is a connected subset of $\Xi$:
if $x,x' \in X^\ell$,
then there is a continuous path $(x_t)_{t \in [0,1]}$ in $X^\ell$ that starts at $x_0 = x$ and ends at $x_1 = x'$.


\subsection{The minmax value}

For every $i\in I$, 
denote by $v_i:=\inf_{\sigma_{-i}\in \Sigma_{-i}}\sup_{\sigma_i\in \Sigma_i} \gamma_i(\sigma)$ player~$i$'s minmax value.
Since payoffs are non-negative and  bounded by $1$, we have $0 \leq v_i\leq 1$ for each $i \in I$.

A \emph{minmaxing profile against $i$} is a profile $(\sigma_i,\sigma_{-i})$ such that 
 $\gamma_i(\widetilde \sigma_i,\sigma_{-i})\leq v_i$ for each $\widetilde \sigma_i \in \Sigma_i$.

In the proof, 
it is convenient to use stationary minmaxing profiles.
The following result is a generalization of \cite{flesch1996recursive} to multiplayer positive recursive absorbing games.

\begin{lemma}
    For each  $i \in I$ there exists a stationary minmaxing profile.
\end{lemma}

\begin{proof}
For every $\lambda \in (0,1]$
denote by $v_{i,\lambda}:= \min_{\sigma_{-i}\in \Sigma_{-i}}\max_{\sigma_i\in \Sigma_i}\gamma_{i,\lambda}(\sigma)$ player~$i$'s minmax value in the $\lambda$-discounted game, 
where $\gamma_{i,\lambda}(\sigma):=\E_\sigma\left[(1-\lambda)^\theta r_i(a^\theta)\right]$. It is well known that there is a stationary profile $x_{-i,\lambda}\in \Xi_{-i}$ that achieves the minimum, and that  $v_i=\lim_{\lambda \to 0} v_{i,\lambda} \geq 0$, see \cite{neyman2003stochastic}. 
Any limit point $x_{-i,0}$ of the family $(x_{-i,\lambda})_{\lambda>0}$ as $\lambda\to 0$ is a stationary minmaxing profile against $i$. Indeed, let
 $x_i\in \Xi_i$ be arbitrary. If the probability of absorption under $(x_i,x_{-i,0}) $ is positive, then 
\[ \gamma_i(x_i,x_{-i,0})=\lim_{\lambda \to 0} \gamma_{i,\lambda}(x_i,x_{-i,\lambda})\leq 
\lim_{\lambda \to 0} v_{i,\lambda} = 
v_i. \]
If the profile $(x_i,x_{-i,0}) $ is nonabsorbing, then $\gamma_i(x_i,x_{-i,0})=0\leq v_i$.
\end{proof}
\medskip

As in much of the literature on dynamic games, our equilibrium constructions will be supported by the threat of switching to a minmaxing profile in case a player does not abide to a prespecified behavior.
\color{black}

\subsection{The Maximal Absorbing Response}

For every mixed action profile $x \in \Xi$, 
the expected absorption probability under $x$ is 
\[ p(x) := \sum_{a \in A}  \left(p(a) \prod_{j \in I} x_j(a_j)\right), \] 
and player~$i$'s expected absorption payoff under $x$ given absorption is
\[ r_i(x) := \sum_{a \in A}  \left(r_i(a) p(a) \prod_{j \in I} x_j(a_j)\right) / p(x), \ \ \ \forall i \in I. \]
The expected absorption payoff is defined only for absorbing mixed action profiles $x$.

Player~$i$'s \emph{payoff under a best absorbing response} to $x_{-i}$ is 
\begin{equation}
\label{equ:rho}
\rho_i(x) := \max \bigl\{ r_i(a_i,x_{-i}) \colon a_i \in A_i, p(a_i,x_{-i}) > 0\bigr\}.
\end{equation}
By convention, the maximum on an empty set is $-\infty$.
Since the absorbing payoffs are positive,
$\rho_i(x) > 0$ as soon as $\rho_i(x) \neq -\infty$.
Note that $\rho_i(x)$ is independent of $x_i$.
Note also that the functions $x \mapsto \rho_i(x)$ are lower semicontinuous:
$\liminf_{k\to \infty} \rho_i(x^{(k)}) \geq \rho_i(x)$ for
every sequence $(x^{(k)})_{k \in \dN}$ that converges to $x$.
This is so because 
when $(x^{(k)})_{k \in \dN}$ is a sequence in $\Xi$ that converges to $x$,
if $p(a_i,x_{-i}) > 0$,
then $p(a_i,x^{(k)}_{-i}) > 0$ for every $k \in \dN$ sufficiently large,
and hence the maximum in Eq.~\eqref{equ:rho} contains fewer 
actions for $x$ than for $x^{(k)}$, for every large $k$.

When $\rho_i(x) > -\infty$, let $a_i(x) \in A_i$ be
a \emph{best absorbing response} to $x_{-i}$,
namely,
an action that attains the maximum in Eq.~\eqref{equ:rho}.

We note that if there is a nonabsorbing mixed action profile such that $\rho_i(x) = -\infty$,
then $v_i = 0$.
Indeed, since the game is positive and recursive, $v_i \geq 0$, 
and by playing the mixed action profile $x_{-i}$,
the other players lower the best payoff player~$i$ can obtain to $0$.

%
%
%
%
%

If there is $x \in X$ such that 
$\rho_i(x) = -\infty$ for every $i \in I$,
then the stationary strategy $x$ is an undiscounted $0$-equilibrium:
it yields each player the payoff 0,
and all unilateral deviations are nonabsorbing
and hence not profitable.
Since our goal is to prove the existence of an undiscounted equilibrium payoff,
we can make the following assumption.

\begin{assumption}[No nonabsorbing equilibrium]
\label{assumption:rho}
For every $x \in X$ there is $i \in I$ such that $\rho_i(x) > 0$.
\end{assumption}


We note%
\footnote{
Assumption~\ref{assumption:rho} and the inequality $\rho_i(x) \geq v_i$ whenever the left-hand side is finite are two properties that are used in the proof and require
the assumption that the game is recursive and positive.}
that $\rho_i(x)\geq v_i$
whenever $\rho_i(x) > 0$.
It will be convenient to denote
\[ \rho(x) := (\rho_i(x))_{i \in I}, \ \ \ \forall x \in X. \]

\subsection{Partitions of the Space of Payoff Vectors}
\label{section:partition}

For every nonabsorbing mixed action profile $x \in X$ 
we partition $\dR^I$ into three subsets, $W(x)$, $W_L(x)$, and $W_H(x)$.
Define (see Figure~\arabic{figurecounter} for an illustration with two players):
\begin{align*}
W_H(x) &:= \{ w \in \dR^I \colon w_i > \rho_i(x), \ \ \ \forall i \in I\},\\
W(x) &:=  \left\{ w \in \dR^I \colon 
\begin{array}{ll}
w_i \geq \rho_i(x) & \forall i \in I,\\
w_i =  \rho_i(x), & \hbox{for some } i \in I
\end{array}
\right\},\\
W_L(x) &:= \{ w \in \dR^I \colon 
w_i < \rho_i(x), \ \ \ \hbox{for some } i \in I\}.
\end{align*}



\centerline{
\begin{tikzpicture}[scale=0.3]
  \fill[gray!30] (9,9) -- (-9,9) -- (-9,-9) -- (9,-9) -- (9,9);

  \draw[dotted] (-10,-2) -- (10.5,-2);
  \draw[dotted] (-2,-10) -- (-2,10.5);
  \draw[->] (9,-2) -- (10.5,-2) node[right] {$w_1$};
  \draw[->] (-2,9) -- (-2,10.5) node[above] {$w_2$};
  \draw[->] (13,0) -- (9.2,-1.6);

  \draw[line width=2pt] (-2,9) -- (-2,-2) -- (9,-2);



  \node at (-5.5, -5.5) {$W_L(x)$};
  \node at (4, 4) {$W_H(x)$};
  \node at (14.5, 0) {$W(x)$};

  \node[below] at (-2,-10.5) {$\rho_1(x)$};
  \node[left] at (-10.5,-2) {$\rho_2(x)$};
\end{tikzpicture}
}

\centerline{Figure~\arabic{figurecounter}: The sets $W(x)$, $W_L(x)$, and $W_H(x)$.}
\addtocounter{figurecounter}{1}

\bigskip

Set
\begin{align*}
W_H := \bigcup_{x \in X} W_H(x), \ \ \ 
W := \bigcup_{x \in X} W(x), \ \ \ 
W_L := \bigcup_{x \in X} W_L(x).
\end{align*} 

We note that the three sets $W$, $W_H \setminus W$,
and $\bigcap_{x \in X} W_L(x)$ form a partition of
$\dR^I$.
Indeed, by De Morgan's laws,
\begin{align*}
\bigcap_{x \in X} W_L(x)
&= \bigcap_{x \in X} (\dR^I \setminus (W_H(x) \cup W(x)))\\
&= \dR^I \setminus \left(\bigcup_{x \in X} \bigl(W_H(x) \cup W(x)\bigr)\right)\\
&= \dR^I \setminus (W_H \cup W).
\end{align*}

\subsection{Exits}

We here present and study the notions of exits and joint exits,
which were introduced by \cite{Solan1999ThreePlayerAbsorbing}.

\begin{definition}[Exit, joint exit]
Let $x \in X$ be a nonabsorbing mixed action profile.
A pair $(J,a_J)$ where $\emptyset \neq J \subseteq I$ and $a_J \in \prod_{i \in J} A_i$ is an \emph{exit at $x$} if 
\begin{itemize}
\item $p(a_J,x_{-J}) > 0$.
\item $p(a_{J'},x_{-J'}) = 0$ for every $J' \subset J$.
\end{itemize}
An exit $(J,a_J)$ is 
\emph{joint} if $|J| \geq 2$.
The set of 
joint exits from $x$ is
denoted 
$\calE_J(x)$.
\end{definition}

For future reference we note that if a sequence of nonabsorbing mixed action profiles that have joint exits converges to a limit,
then the limit also has a joint exit.

\begin{lemma}
\label{lemma:joint:continuous}
If $(x^{(k)})_{k \in \dN}$ is a sequence of nonabsorbing mixed actions that converges to $x$,
and if $\calE_J(x^{(k)}) \neq \emptyset$,
then $\calE_J(x) \neq \emptyset$.
\end{lemma}

The following observation relates joint exits to rectangular connected components. We omit the proof as well.

\begin{lemma}
\label{lemma:joint:rectangular}
The connected component $B^\ell$ is rectangular
if and only if there is no joint exit at any mixed action profile in $X^\ell$.
\end{lemma}


\begin{example}
\label{example:1}
Consider the positive recursive absorbing game in Figure~\arabic{figurecounter},
where Player~1 has $3$ actions, Player~2 has $4$ actions,
empty entries are nonabsorbing,
and each nonempty entry $a$ is absorbing with positive probability and contains the vector $r(a)$. 
The four pairs $(\{1,2\}, (a'_1,a'_2))$, $(\{1,2\}, (a'_1,a''_2))$, $(\{1,2\}, (a''_1,a'_2))$, $(\{1,2\}, (a''_1,a''_2))$
are joint exits from $(a_1,a_2)$.
The pair $(\{2\}, (a_1,a'''_2))$ is an exit of $(a_1,a_2)$ but not a joint exit.
There is no joint exit at $(a'_1,a_2)$.

\centerline{
\begin{tikzpicture}
  \def\boxheight{1}
  \def\boxlen{1.3}

  \foreach \row in {1,2,3} {
    \foreach \col in {1,2,3,4} {
      \draw[thick] (\col*\boxlen, -\row*\boxheight) rectangle ++(\boxlen, \boxheight);
    }
  }

  \node at (4.5*\boxlen, -.5*\boxheight) {$(\frac{1}{2},\frac{1}{4})$};
  \node at (4.5*\boxlen, -1.5*\boxheight) {$(\frac{1}{3},\frac{3}{5})$};
  \node at (4.5*\boxlen, -2.5*\boxheight) {$(\frac{2}{3},\frac{2}{3})$};
  \node at (2.5*\boxlen, -1.5*\boxheight) {$(\frac{1}{4},\frac{1}{3})$};
  \node at (3.5*\boxlen, -1.5*\boxheight) {$(\frac{1}{2},\frac{1}{2})$};
  \node at (2.5*\boxlen, -2.5*\boxheight) {$(\frac{1}{5},\frac{3}{4})$};
  \node at (3.5*\boxlen, -2.5*\boxheight) {$(\frac{2}{5},\frac{3}{5})$};


  \node[left] at (1*\boxlen, -.5*\boxheight) {$a_1$};
  \node[left] at (1*\boxlen, -1.5*\boxheight) {$a'_1$};
  \node[left] at (1*\boxlen, -2.5*\boxheight) {$a''_1$};

  \node[above] at (1.5*\boxlen, -0*\boxheight) {$a_2$};
  \node[above] at (2.5*\boxlen, -0*\boxheight) {$a'_2$};
  \node[above] at (3.5*\boxlen, -0*\boxheight) {$a''_2$};
  \node[above] at (4.5*\boxlen, -0*\boxheight) {$a'''_2$};
\end{tikzpicture}}

\centerline{Figure~\arabic{figurecounter}: The positive recursive absorbing game in Example~\ref{example:1}.}
\addtocounter{figurecounter}{1}
\end{example}

The following result states that at all discontinuity points of $x \mapsto \rho(x)$ there is a joint exit.

\begin{lemma}
\label{lemma:exit:rho}
If $\rho(\cdot)$ is not continuous at $x\in X$, then $\calE_J(x)\neq \emptyset$.
\end{lemma}



\begin{proof}
Suppose that there is no joint exit at $x$.
We will show that $\rho_i(x) = \lim_{k \to \infty} \rho_i(x^{(k)})$ for every $i \in I$.

For every $y \in X$ denote the set of all absorbing responses of player~$i$ to $y_{-i}$ by
\[ \widehat A_i(y) := \{a_i \in A_i \colon p(a_i,y_{-i}) > 0\}. \]
Assume w.l.o.g.~that $\supp(x^{(k)})$ is independent of $k$,
and therefore $\widehat A_i(x^{(k)})$ is independent of $k$.
In this proof we will denote this set simply by $\widehat A_i(x^{(k)})$.

Fix $i \in I$.
To prove that $\rho_i(x) = \lim_{k \to \infty} \rho_i(x^{(k)})$ it is sufficient to show that 
$\widehat A_i(x) = \widehat A_i(x^{(k)})$.

We note that the inequality $\widehat A_i(x) \subseteq \widehat A_i(x^{(k)})$ holds because $\supp(x^{(k)}) \supseteq \supp(x)$ for all $k \in \dN$.
We will next prove the reverse inclusion.

We argue by contradiction and assume that there is $a_i\in  \widehat A_i(x^{(k)})\setminus \widehat A_i(x)$. Since 
$p(a_i,x_{-i}^{(k)}) > 0$ for all $k \in \dN$,
while $p(a_i,x_{-i}) = 0$,
there is an action profile $\widehat a_{-i} \in \supp(x^{(k)}_{-i})$ for all $k \in \dN$,
such that $p(a_i,\widehat a_{-i}) > 0$.
Denoting by $J \subseteq I \setminus \{i\}$ the set of all players $j$ such that $\widehat a_j \not\in \supp(x_j)$,
we obtain that $|J| \geq 1$ and $(J \cup \{i\}, (a_i,\widehat a_J))$ is a joint exit at $x$.
%
\end{proof}

%
%
%
%


\bigskip




The next result relates joint exits to the set $W_H(x)$:
it states that if a payoff vector $w$ lies in $W_H \setminus W$,
then there is $x \in X$ such that not only does $w$ lie in $W_H(x)$, but also $x$ has a joint exit.

\begin{lemma}
\label{lemma:WH}
If $w \in W_H \setminus W$, 
then there are $\ell \in [L]$ and $x \in X^\ell$ such that $w \in W_H(x)$ and $\calE_J(x) \neq \emptyset$.
\end{lemma}

\begin{proof}
Since $w \in W_H \setminus W$,
there are $\ell \in [L]$ and $x' \in X^\ell$ such that $w \in  W_H(x') \black \setminus W$.
If $\calE_J(x') \neq \emptyset$, we are done.
Assume then that $\calE_J(x') = \emptyset$.

Since $B^\ell$ is not rectangular,
there is a mixed-action profile in $X^\ell$ at which there is a joint exit.
Since $X^\ell$ is connected, 
there is a continuous path $(x_t)_{t \in [0,1]}$ in $X^\ell$ that starts at $x'$ and ends at some $x'' \in X^\ell$ such that $\calE_J(x'') \neq \emptyset$.
By Lemma~\ref{lemma:joint:continuous} we can furthermore assume that $x''$ is the \emph{only} point along the path where $\calE_J(x_t)$ is nonempty.

By Lemma~\ref{lemma:exit:rho},
the function $\rho$ is continuous along the path,
except, possibly at its end.
Since $w \in W_H(x')$,
it follows that $w_i > \rho_i(x_0)$ for all $i \in I$.
Since $w \not\in W(x_t)$ for all $t \in [0,1)$ and since $t \mapsto \rho(x_t)$ is continuous on $[0,1)$,
it follows that $w_i > \rho_i(x_t)$ for all $t \in [0,1)$.
Since $\rho_i$ is lower semicontinuous,
this inequality holds also at $t=1$.
This implies that $w \in W_H(x'')$,
which is what we want to show.
\end{proof}

\bigskip

The following result,
which is a special case of Lemma~5.3 in \cite{Solan1999ThreePlayerAbsorbing},
asserts that if there exists a joint exit that yields all players a high payoff,
then an undiscounted equilibrium payoff exists.

\begin{lemma}[\cite{Solan1999ThreePlayerAbsorbing}, Lemma 5.3]
\label{lemma:multi}
Suppose there exists a nonabsorbing mixed action profile $x \in X$ and a joint exit $(J,a_J) \in \calE_J(x)$
such that 
$r_i(a_J,x_{-J}) \geq \rho_i(x)$ for each $i \in I$.
Then $r(a_J,x_{-J})$ is an undiscounted equilibrium payoff.
\end{lemma}


For a given $\ep > 0$,
one undiscounted $\ep$-equilibrium 
under the conditions of Lemma~\ref{lemma:multi} is the following. 
Until some player is declared the deviator,
the players select the mixed action profile 
$\xi$ defined by
\begin{equation}
\label{equ:xi}
\xi_i := 
\left\{
\begin{array}{ll}
x_i, & i \not\in J,\\
(1-\beta)x_i + \beta a_i, & i \in J,
\end{array}
\right.
\end{equation}
where $\beta > 0$ is sufficiently small.
Thus, if no player is declared the deviator, 
the play is eventually absorbed, and the absorbing payoff is $r(a_J,x_{-J})$.

Since $r_i(a_J,x_{-J}) \geq \rho_i(x)$ for each $i \in I$,
no player $i$ can profit more than $\beta$ by deviating in an absorbing way;
namely, by selecting an action $a_i\in A_i$ such that $a_i \not\in \supp(x_i)$ and $p(a_i,x_{-i}) > 0$.
Since the game is positive recursive, no player in $J$ can profit by never selecting $a_i$.

Players not in $J$ can still profit by changing the frequency in which they select actions in $\supp(x_i)$.
Such a deviation can be identified by a standard statistical test.
For more details, see the proof of Lemma~5.3 in \cite{Solan1999ThreePlayerAbsorbing}.

%


\section{Proof of Theorem~\ref{theorem:uniform}}
\label{section:proof}

If the condition of Lemma~\ref{lemma:multi} holds,
then the game admits an undiscounted equilibrium payoff.
Assume then that this condition does not hold.

\subsection{The Auxiliary One-Shot Game $G(w)$}
\label{section:auxiliary:one:shot}

For every vector $w\in \dR^I$
denote by $G(w)$ the one-shot game derived from $\Gamma$ by assuming that if the game does not absorb in the first stage,
then the players get a terminal payoff $w$. Formally, $G(w)$ is the following one-shot game:
\begin{itemize}
\item 
The set of players is $I$, as in $\Gamma$.
\item 
The set of actions of each player $i \in I$ is $A_i$, as in $\Gamma$.
\item 
The payoff function of each player $i \in I$ is
\[ u^w_i(a) := p(a) r_i(a) + (1-p(a)) w_i, \ \ \ \forall a \in A. \]
\end{itemize}

\subsection{Definition of Two Functions $f$ and $\mu$}
\label{section:def:f}

Fix $\ep \in (0,1)$, and set $Y:=\prod_{i \in I} [v_i-\ep,1]$. The set $Y$ contains all payoff profiles that are approximately individually rational for the players.
In this section we define two functions, 
$f : Y \to Y$
and $\mu : Y \to (0,1]$.

Recall that the three sets
$W$, $W_H \setminus W$,
and $\bigcap_{x \in X} W_L(x)$ form a partition of $\dR^I$.
For $w\in Y$, the definition of $f(w)$ and $\mu(w)$ depends on which of these three sets contains $w$.

In each of the following three sections we define $f$ on one of these sets.
For every nonabsorbing mixed action profile $x \in X$ and every player $i \in I$, denote by
$a_{i}(x)$ a best absorbing response of player~$i$ at $x$.

Fix $w\in Y$. 
\subsubsection{Definition of $f$ and $\mu$ on $W \cap Y$}
\label{subsection:f:W}

Assume  $w \in  W \cap Y$.
Then there is $x \in X$ such that 
$w \in W(x)$.
It follows that 
$w_i \geq \rho_i(x)$ for each $i \in I$,
with equality for at least one player, denoted $i_0$.
Set
\begin{equation}
\label{equ:f:w}
f(w) := \ep p(a_{i_0}(x),x_{-i_0}) r(a_{i_0}(x),x_{-i_0})
+ (1-\ep p(a_{i_0}(x),x_{-i_0})) w,
\end{equation}
and
\[ \mu(w) := \ep p(a_{i_0}(x),x_{-i_0}). \]

The quantity $f(w)$ is the payoff in $G(w)$ under the strategy profile
where player~$i_0$ selects the mixed action $(1-\ep)x_{i_0}+\ep a_{i_0}(x)$,
and the other players select the mixed action profile $x_{-i_0}$.

We argue that $f(w) \in Y$.
Fix for a moment $i \in I$.
If $\rho_i(x) > -\infty$,
then $\rho_i(x) \geq v_i > 0$.
Since $w_i\geq \rho_i(x)\geq v_i$ and since $\|f(w)-w\|_\infty\leq \ep$, 
it follows that $f_i(w) \geq v_i - \ep$.
Suppose next that $\rho_i(x) = -\infty$, so that $v_i = 0$.
Since the game is positive, 
$f_i(w) \geq (1-\ep)w_i \geq -(1-\ep)\ep > -\ep = v_i - \ep$.
Since $i$ is arbitrary, $f(w) \in Y$.
\black

\subsubsection{Definition of $f$ and $\mu$ on  $(W_H \setminus W) \cap Y$}
\label{subsection:f:WH}

Assume $w \in  (W_H \setminus W) \cap Y$.
By Lemma~\ref{lemma:WH},
there is $x \in X$ such that $w \in W_H(x)$ and $\calE_J(x) \neq \emptyset$.
Let $(J,a_J) \in \calE_J(x)$.
Since the conditions of Lemma~\ref{lemma:multi} do not hold,
there is $i_0 \in I$ such that $r_{i_0}(a_J,x_{-J}) < \rho_{i_0}(x) - \ep$.

Since $w \in W_H(x)$, we have
$w_i > \rho_i(x)$ for every $i \in I$.
Let $\alpha \in (0,1)$ be the unique real number such that
$(1-\alpha) w_i + \alpha r_i(a_J,x_{-J}) \geq \rho_i(x)$ for every $i \in I$,
with an equality for at least one $i$.

Define
\begin{equation}
\label{equ:f:wh}
f(w) := \alpha r(a_J,x_{-J}) + (1-\alpha) w, 
\end{equation}
and 
\[ \mu(w) = \alpha. \]
Note that $f(w)\geq \rho(x)\geq v$ (component-wise), and hence $f(w)\in Y$.

The quantity $f(w)$ is the payoff in $G(w)$
when with probability $\alpha$ the play is absorbed through $(a_J,x_{-J})$,
and with probability $1-\alpha$ the game is not absorbed.

\subsubsection{Definition of $f$ and $\mu$ on 
 $\bigl(\bigcap_{x \in X} W_L(x)\bigr) \cap Y$}
\label{subsection:f:WL}

Assume $w\in  \bigl(\bigcap_{x \in X} W_L(x)\bigr) \cap Y$. 
Let $\widehat x$ be a Nash equilibrium of the auxiliary one-shot game $G(w)$.
We argue that
$p(\widehat x) > 0$.
Indeed, suppose by contradiction that $p(\widehat x) = 0$,
so that $\widehat x \in X$.
Since $w \in \bigcap_{x \in X}W_L(x) \subseteq W_L(\widehat x)$,
there is $i \in I$ such that $w_i < \rho_i(\widehat x)$.
But then player~$i$ can profit by deviating at $\widehat x$ to $a_i(\widehat x)$,
since in $G(w)$ her payoff under $\widehat x$ is $w_i$,
while her payoff under $(a_i(\widehat x),\widehat x_{-i})$
is a convex combination of $w_i$ and $\rho_i(\widehat x)$.
This contradicts the fact that $\widehat x$ is a Nash equilibrium of $G(w)$.

Set
\[ f(w) := u^w(\widehat x), \]
and
\[ \mu(w) := p(\widehat x). \]

We argue that $f_i(w) \geq v_i - \ep$ for each player $i \in I$, and hence $f(w) \in Y$.
To this end, we will show that player~$i$'s
minmax value in $G(w)$ is at least $v_i - \ep$.
It is well known that for every 
$\lambda \in (0,1]$,
player~$i$'s minmax value in $G((1-\lambda) v_{\lambda})$ is $v_{i,\lambda}$.
By continuity of the minmax value operator,
player~$i$'s minmax value in $G(v)$ is $v_i$.
Since the minmax value operator is nondecreasing and nonexpansive,
and since $w_i \geq v_i - \ep$,
player~$i$'s minmax value in $G(w)$ is at least $v_i - \ep$.

\subsection{Constructing Equilibria in Finite Stage Games}

In this section, we explain how to implement $f(w)$ as an approximate equilibrium payoff in a finite-stage version of $\Gamma$, in which the terminal payoff is $w\in Y$ in case the game does not absorb. In Section~\ref{section:construction} we explain how to concatenate this construction over several blocks.

Given $T\in \dN$ and $w\in Y$, we denote by $G_T(w)$ the $T$-stage game that is derived from $\Gamma$ by assuming that the players get a terminal payoff $w$ if the game does not absorb in the first $T$ stages. The game $G(w)$ introduced in Section~\ref{section:auxiliary:one:shot} thus coincides with $G_1(w)$.

The following result states that for every $w \in Y$ there is $T \in \dN$ sufficiently large,
such that in $G_T(w)$ there is an approximate equilibrium with payoff close to $f(w)$.
The statement of this lemma involves a
stopping time $\tau_i$,
which is interpreted as the first stage in which player $i$ is identified as a deviator. 

\begin{lemma}\label{lemm building block}
For each $w\in Y$ and $\eta>0$, there is $T = T(w,\eta) \in \dN$ such that the game $G_T(w)$ has the following property: There is a stationary strategy profile $\sigma = \sigma(w,\eta)$ and a family $(\tau_i)_{i\in I}$ of stopping times such that: 
\begin{itemize}
\item $\prob_\sigma\left(\min_{i\in I}\tau_i \leq T\right)\leq \eta$.
\item $f(w)=\E_\sigma\left[ r(a^\theta) 1_{\theta \leq T} + w 1_{\theta >T}\right]$  and $\mu(w)= \prob_\sigma(\theta \leq T)$.
\item For each player $i$ and each alternative strategy $\widetilde \sigma_i\in \Sigma_i$, 
\end{itemize}
\begin{align*}
\E_{\widetilde \sigma_i,\sigma_{-i}}&\left[ r_i(a^\theta) 1_{\theta \leq \min\{T, \tau_i\}} + w_i \mathbf{1}_{T< \min\{\theta,\tau_i\}}+v_i1_{\tau_i\leq \min\{T,\theta-1\}}\right]\\ 
&\leq  f(w) +\eta + \ep \prob_{\widetilde \sigma_i,\sigma_{-i}}\left(\tau_i \leq \min\{T,\theta-1\}\right).\end{align*}
\end{lemma}


The first property states that under $\sigma$, it is highly unlikely that some player will be incorrectly identified as a deviator. The second property states that, ignoring $(\tau_i)_{i \in I}$,  the payoff and probability of absorption induced by $\sigma$ coincide with $f(w)$ and $\mu(w)$. 
The last property bounds the expected gain obtained by deviating, assuming that a deviator is punished at her minmax value if it is detected.
The formulation of the last property is important when concatenating several such finite-stage games, to avoid small gains from accumulating into a larger one. 

\bigskip

\begin{proof}
Fix $w\in W$ and $\eta>0$. 
The definition of $\sigma$ and $T$ will reflect the construction of $f$. 

\paragraph{Case 1: $w\in \bigcap_{x \in X} W_L(x)$.}
In that case, there is $\widehat x \in X$ such that $p(\widehat x) > 0$ and $\widehat x$ is a Nash equilibrium of $G(w)$ with payoff $f(w)$ (see Section~\ref{subsection:f:WL}).
We set $T=1$, $\sigma= \widehat x$, 
and no statistical test is employed: $\tau_i= +\infty$ for every $i \in I$.

\paragraph{Case 2: $w\in W$.}
In that case, there is a nonabsorbing mixed action profile $x \in X$ and a player $i_0 \in I$ such that $w \in W(x)$ and
\[ f(w) = (1-\ep p(a_{i_0}(x),x_{-i})) w + \ep p(a_{i_0}(x),x_{-i}) r(a_{i_0}(x),x_{-i}), \]
(see Section~\ref{subsection:f:W}).

The profile $\sigma$ plays repeatedly the mixed action profile $y$ defined by
\[ y_i := \left\{
\begin{array}{ll}
x_i, &i \neq i_0,\\
(1-\beta)x_i + \beta a_{i_0}(x), & i = i_0,
\end{array}
\right.
\]
where $\beta$ is specified below.

Given $\beta$, the duration $T\in \dN$ of the game is such that 
the total probability that the game is absorbed is $\ep$:
\[ (1-\beta p(a_{i_0}(x),x_{-i}))^{T} = 1-\ep. \]

Against $\sigma_{-i}$, a player $i$ may potentially benefit by
\begin{itemize}
    \item increasing or lowering the probability of ever playing $a_{i_0}$, if $i=i_0$;
    \item modifying the relative frequencies of actions in the support of $x_i$, for $i\neq i_0$;
    \item choosing an action not in the support of $a_i$, for any $i\in I$.
\end{itemize}

The first type of deviation is not profitable, given the properties of $a_{i_0}$; the second type of deviation can be monitored by checking that the empirical distribution of each player $i$'s choices remains close to $x_{i}$; the last type of deviation is immediately detected. 

The stopping time $\tau_i$ deals with the latter two types of deviation and is defined to be the first stage in which player $i$ chooses an action which is not in the support of $y_i$, or in which the empirical distribution of player $i$'s actions is not within some appropriate confidence set around $x_i$, that depends on $\eta$. It is set to $\tau_i=+\infty$ otherwise, if both tests are passed.

The precise definition of $(\tau_i)_{i \in I}$ uses standard statistical tests
which appear in, e.g., the proof of Lemma 5.3 in \cite{Solan1999ThreePlayerAbsorbing} or in the proof of Proposition 7 in \cite{vieille2000recursive}. For that reason, we omit details.
\medskip

If $\beta$ is small enough, and for an appropriate confidence set, $\sigma$ and $T$ satisfy the conclusions of Lemma \ref{lemm building block}: (i) undetected deviations modify the distribution of moves by at most $\eta$ and therefore do not affect payoffs by more than $\eta$; (ii) since $f_i(w)\geq \rho_i(w)-\ep\geq v_i-\ep$ for each $i$, detected deviations do not improve payoffs by more than $\ep$.

\paragraph{Case 3: $w\in W_H\setminus W$.}
In that case, there is $x \in X$, a joint exit $(J,a_J) \in \calE_J(x)$, and $\alpha \in (0,1)$, such that $w \in W_H(x)$ and
\[ f(w) = \alpha r(a_J,x_{-J}) + (1-\alpha) w, \]
(see Section~\ref{subsection:f:WH}).

The profile $\sigma$ plays repeatedly the mixed action profile $y$ defined by
\[ y_i := \left\{
\begin{array}{ll}
x_i, &i \not\in J,\\
(1-\beta)x_i + \beta a_i, & i \in J,
\end{array}
\right.
\]
where $\beta$ will be specified below. Given $\beta$, the duration $T$ of the block is set to be such that the total probability that the game is absorbed during the block is $\alpha$:
\[ (1- \beta^{|J|} p(a_J,x_{-J}))^{T} = 1-\alpha. \]

Against $\sigma_{-i}$, a player $i$ may potentially benefit by
\begin{itemize}
    \item modifying the frequency of $a_{i}$, if $i\in J$;
    \item modifying the relative frequencies of actions in the support of $x_i$, for $i\notin J$;
    \item choosing an action not in the support of $a_i$, for any $i\in I$.
\end{itemize}

The latter two types of deviation are prevented as in \textbf{Case 2}, but the former one is different. For $\beta$ small, the definition of $T$ ensures that $T\beta$ is of the order of $\displaystyle \frac{1}{\beta^{|J|-1}}$ up to constants, so that the expected number of stages in which player $i\in J$ chooses $a_i$ increases to $+\infty$ as $\beta \to 0$.
This fact allows players $-i$ to monitor the empirical frequency with which $a_i$ is chosen. We refer to \cite{vieille2000recursive}, Section 8.2 for details.
\end{proof}

\subsection{Constructing a Strategy profile $\sigma^*$}
\label{section:construction} 

We use the function $f$ to construct an approximate equilibrium  $\sigma^*$.
Our construction is inspired by \cite{SolanVieille2001QuittingGames},
see also Section 12.3 in \cite{solan2022course}. 
The play will be divided into blocks. On each block, the strategy profile $\sigma^*$ is obtained by applying  Lemma \ref{lemm building block} to some continuation payoff $w^{(k)}$, as long as the game has not absorbed and no deviator has been identified.

The construction proceeds as follows.
In Step~1 we use an approximate finite orbit of $f$,
to construct a finite sequence $(w^{(k)})_{k=1}^{K_0}$
such that $f(w^{(k+1)})$ is close to $w^{(k)}$.
In Step~2 we will use this sequence to construct a strategy profile in blocks,
such that for each $k=0,\dots,K_0$,
the expected payoff from the beginning of block $k$ will be close to 
$w^{(k)}$.
Within block $k$, 
the strategy profile $\sigma^*$ will 
be derived from $f(w^{(k+1)})$ using Lemma \ref{lemm building block}

In Section~\ref{section:equ:proof} we will prove that $\sigma^*$ is an undiscounted $7\ep$-equilibrium.

\bigskip
\noindent\textbf{Step 1:}
Definition of a sequence of (approximate) continuation payoffs.


Let $\delta \in (0,-\frac{1}{\ln(\ep)})$, so that $\exp(-1/\delta) < \ep$, and hence in particular $\delta < \ep$.
By Theorem~4 in \cite{SolanSolan2020QuittingLCP}, see also Theorem~12.6 in \cite{solan2022course}, 
$f$ has an approximate orbit
that starts at the vector of maximal payoffs.\footnote{Both \cite{SolanSolan2020QuittingLCP} and \cite{solan2022course} prove this result when $f$ has no fixed point.
However, the result also holds when $f$ has fixed points.
To see this, 
note that if $w^{(k)}$ is not $\delta$-close to any fixed point of $f$, for all $k$,
then the proof in \cite{SolanSolan2020QuittingLCP} and \cite{solan2022course} is valid.
If $w^{(k)}$ is $\delta$-close to some fixed point $\widetilde w$ of $f$,
set $w^{(l)} = \widetilde w$ for every $l < k$.
We then obtain \eqref{equ:periodic0}--\eqref{equ:periodic2} with $2\delta$ instead of $\delta$.} 
Formally, this theorem implies the existence of a finite sequence $(w^{(k)})_{k=0}^{K_0}$ that satisfies the following properties:
\begin{align}
\label{equ:periodic0}
&w^{(K_0)} := (1,1,\dots,1),\\
\label{equ:periodic1}
&\sum_{k=0}^{K_0-1} \mu(w^{(k)}) \geq \frac{1}{\delta},\\
&\sum_{k=0}^{K_0-1} \bigl\| w^{(k)} - f(w^{(k+1)})\bigr\|_\infty \leq \delta. 
\label{equ:periodic2}
\end{align}

\noindent\textbf{Step 2:}
Definition of $\sigma^*$.

Fix $\eta^{(0)},\ldots, \eta^{(K_0-1)}$  such that $\displaystyle \sum_{k=0}^{K_0-1} \eta^{(k)} <\ep$. 
For each $k$, 
$0 \leq k \leq K_0-1$, 
we apply Lemma \ref{lemm building block} with $w= w^{(k+1)}$ and $\eta= \eta^{(k)}$. 
We denote by 
$T^{(k)}$, 
$\sigma^{(k)}$ and $(\tau_i^{(k)})_{i\in I}$ 
the horizon,
the stationary profile, and the stopping times for which the conclusion of Lemma \ref{lemm building block} holds.

The strategy profile $\sigma^*$ is defined in $K_0$ blocks. The length of block $k$ is $T^{(k)}$. The cumulative duration of the first $k$ blocks is $\displaystyle N^{(k)}:=\sum_{m=0}^{k-1} T^{(m)}$.

The players follow the sequence $(\sigma^{(k)})_{k=0}^{K_0-1}$ of stationary profiles in the successive blocks as long as no deviator has been identified;
that is, as long as the stopping times $(\tau_i^{(k)})_{i \in I}$ are not reached.
The behavior of the players on each block is monitored by means of the  stopping times $(\tau^{(k)}_i)$. We denote by $\tau_i$ the first stage $n\leq N^{(K_0-1)}$ in which player $i$ is identified as a deviator,
and by $\tau := \min_{i \in I} \tau_i$ the first stage in which some player is identified as the deviator.
Whenever $\tau < \infty$ we let $i_*$ be the minimal index (according to some fixed order of the players) such that $\tau_{i_*} = \tau$.
At stage $\tau$,
players $-i_*$
switch to a stationary minmaxing profile against $i_*$.
The definition of $\sigma^*$ beyond block $K_0$, in the event where no deviation has been detected, is irrelevant.

Eq.~\eqref{equ:periodic1} will ensure that under $\sigma^*$ the play is absorbed with probability close to $1$.
Eq.~\eqref{equ:periodic2} will ensure that the fact that $w^{(k)}$ is not necessarily equal to $f(w^{(k+1)})$ does not distort the players' payoffs too much.

\subsection{$\sigma^*$ is an undiscounted $7\ep$ equilibrium}
\label{section:equ:proof}

For $k=0,\ldots, K_0$, we define 
\[Z^{(k)}:= \left\{ \begin{array}{ll}
r(a^\theta), & \mbox{ if } \theta \leq \min\{N^{(k-1)}, \tau\}, \\
w^{(k)}, & \mbox{ if } N^{(k-1)} < \min\{\theta,\tau\}, \\
v, &\mbox{ if } \tau \leq \min\{N^{(k-1)}, \theta -1\},\end{array}\right.\]
(with $N^{(-1)}=0$). The random variable  $Z^{(k)}$ can be thought of as the expected undiscounted payoff, conditional on the history up to block $k$. If the game has absorbed, the undiscounted payoff is $r(a^\theta)$; the undiscounted payoff is otherwise $w^{(k)}$ if no deviator has been identified.

Denote by $\calH^{(k)}$ the $\sigma$-algebra induced by histories up to the beginning of block $k$. 
Since 
the difference between $f(w^{(k+1)})$ and $\E_{\sigma^{(k)}} \left[ w^{(k+1)} \right]$ stems from the stopping times $(\tau_i^{(k)})$,
the following statement is a rewriting of the properties of $\sigma^{(k)}$ and $\tau^{(k)}$.

\begin{lemma}\label{lemm estimate}
    Under $\sigma^*$, one has with probability 1,
    \begin{equation}\label{approx}\left\|\E_{\sigma^*}\left[ Z^{(k+1)} \mid \calH^{(k)}\right] -Z^{(k)}\right\|_\infty \leq 
    \eta^{(k)} + \|w^{(k)}-f(w^{(k+1)})\|_\infty, \ \ \  \forall k=0,\dots,K_0-1.\end{equation}

Moreover, for each player $i\in I$ and each alternative strategy 
$\widetilde \sigma_{i}$,
\begin{eqnarray}
\nonumber
\E_{\widetilde \sigma_i,\sigma^*_{-i}}\left[Z_i^{(k+1)} \mid \calH^{(k)} \right] &\leq &Z_i^{(k)} + \ep \prob_{\widetilde \sigma_i,\sigma^*_{-i}}\left( N^{(k-1)}\leq \tau_i< N^{(k)}\mid \calH^{(k)}\right) \\
    && + \eta^{(k)} + \left|w^{(k)}_i-f_i(w^{(k+1)})\right|.
    \label{equ:9.1}
\end{eqnarray}
\end{lemma}

The bound (\ref{approx}) implies that $Z^{(0)}$ is approximately equal to $\gamma(\sigma^*)$.
Indeed, taking first expectations and then using the triangular inequality, one obtains
\begin{align}
\nonumber
\left\| Z^{(0)}-\E_{\sigma^*}\left[Z^{(K_0)}\right]\right\|_\infty 
&\leq 
\sum_{k=0}^{K_0-1}\black \eta^{(k)} + \sum_{k=0}^{K_0-1}\black \|w^{(k)}- f(w^{(k+1)})\|_\infty\\
&\leq \ep +\delta \leq 2\ep.
\label{equ:10.0}
\end{align}
On the other hand, since the function $\ln(1-x)$ is concave on $[0,1)$,
Eq.~\eqref{equ:periodic1} implies that the 
probability that the game terminates under $\sigma^*$ in the first $K_0$ blocks is at least $1-\exp(-1/\delta)> 1-\ep$.
Since the probability that some player is incorrectly labeled as a deviator is at most $\ep$, 
this implies that 
\begin{equation}
\label{equ:10.1}
\left\| \E_{\sigma^*}\left[Z^{(K_0)}\right] - \gamma(\sigma^*)\right\|_\infty \leq \ep +\ep = 2\ep.
\end{equation}

We next show, using
Eq.~\eqref{equ:9.1},
that no player can improve significantly upon $\sigma^*$. Let $i\in I$ and $ \sigma_i\in \Sigma_i$ be arbitrary. Taking expectations and summing over $k$
in Eq.~\eqref{equ:9.1},
one has 
\begin{equation}\label{dev}\E_{\sigma_i,\sigma^*_{-i}}\left[ Z_i^{(K_0)}\right]\leq Z_i^{(0)} + \ep + \sum_{k=0}^{K_0-1}\black \eta^{(k)} 
+\sum_{k=0}^{K_0-1}\black\left|w^{(k)}_i-f_i(w^{(k+1)})\right| \leq Z_i^{(0)}+3 \ep. \end{equation}
Observe next that the expected payoff induced by $(\sigma_i,\sigma^*_{-i})$ after the $K_0$-th block is at most 1 in the event where $N^{(K_0)}< \min\{\theta,\tau\}$, which implies that $\gamma_i(\sigma_i,\sigma^*_{-i})\leq \E_{\sigma_i,\sigma^*_{-i}}\left[ Z_i^{(K_0)}\right]$. 
By (\ref{dev}), \eqref{equ:10.0}, and~\eqref{equ:10.1},
\[\gamma_i(\sigma_i,\sigma^*_{-i})\leq Z_i^{(0)} +
3\ep \leq 
\E_{\sigma^*}\left[Z^{(K_0)}\right] + 5\ep \leq 
\gamma_i(\sigma^*) + 7\ep.\]
Thus, $\sigma^*$ is an undiscounted $7\ep$-equilibrium.

\section{Discussion and Open Problems}
\label{section:discussion}

As mentioned before, \cite{SolanSolan2021SunspotEquilibrium} proved that an undiscounted equilibrium payoff exists when 
(i) there is exactly one connected component which in rectangular, 
(ii) 
each player has exactly one action that is not a part of any nonabsorbing entry,
and
(iii)
the probability of absorption in all entries is either $0$ or $1$.
A natural question is whether this result can be combined with ours, 
to prove that every positive recursive absorbing game that does not have a rectangular connected component $B^\ell$ in which $|B^\ell_i| \geq 2$ for at most one player admits an undiscounted equilibrium payoff.
At present, we do not know how to prove this result.
The reason is that the proof approach of \cite{SolanSolan2021SunspotEquilibrium} and our approach are different, and it is not clear how to unify these approaches.

A second open problem concerns the extension of Theorem~\ref{theorem:uniform} to recursive absorbing games that are not positive and to absorbing games that are not recursive.
These extensions pose difficulties because, when the game is not positive recursive,
$\rho_i(x)$ may be finite but not higher than player~$i$'s minmax value $v_i$.
In that case, the cutoff for defining the sets $W(x)$, $W_H(x)$, and $W_L(x)$ should be $\max\{\rho_i(x),v_i\}$,
and then it is not clear how to define $f$ when $w \in W(x)$ and the only player $i$ for whom $w_i = \max\{\rho_i(x),v_i\}$ satisfies $w_{i} = v_{i} > \rho_{i}(x)$.

It would also be interesting to extend Theorem~\ref{theorem:uniform} to positive recursive games,
which may include more than a single nonabsorbing state.


\bibliographystyle{chicago}
\bibliography{biblio.bib,solan_publications}


\end{document}